\documentclass[12pt]{amsart}

\usepackage{amsmath,amssymb,amsthm}

\usepackage{booktabs}
\usepackage[letterpaper,margin=1in,ignoreall]{geometry}
\usepackage{hyperref}

\newtheorem{theorem}{Theorem}[section]
\newtheorem{lemma}[theorem]{Lemma}

\newtheorem{cor}[theorem]{Corollary}
\newtheorem*{apolaritylemma}{Apolarity Lemma}

\theoremstyle{remark}
\newtheorem{example}[theorem]{Example}
\newtheorem{remark}[theorem]{Remark}
\newtheorem{notation}[theorem]{Notation}

\renewcommand{\d}{\partial}
\newcommand{\C}{\mathbb{C}}
\newcommand{\PP}{\mathbb{P}}
\newcommand{\SL}{\mathrm{SL}}
\newcommand{\SO}{\mathrm{SO}}

\newcommand{\GL}{\mathrm{GL}}
\DeclareMathOperator{\Span}{span}
\DeclareMathOperator{\Diff}{Diff}
\DeclareMathOperator{\dett}{det}
\DeclareMathOperator{\pf}{pf}

\DeclareMathOperator{\sdet}{sdet}

\newcommand{\wtV}{\widetilde{\mathbb V}}
\newcommand{\VV}{{\mathbb V}}
\newcommand{\wtH}{\widetilde{H}}

\newcommand{\defining}[1]{\textbf{#1}}

\title{Lower bound for ranks of invariant forms}

\date{\today}

\author{Harm Derksen}
\address{Harm Derksen \\
Department of Mathematics \\
University of Michigan \\
530 Church Street \\
Ann Arbor, MI 48109--1043}
\email{hderksen@umich.edu}

\author{Zach Teitler}
\address{Zach Teitler \\
Boise State University \\
Department of Mathematics \\
1910 University Drive \\
Boise, ID 83725--1555}
\email{zteitler@boisestate.edu}


\thanks{The first author was supported by NSF grant DMS 1302032.}
\begin{document}
\begin{abstract}
We give a lower bound for the Waring rank and cactus rank of forms that are invariant under an action of a connected algebraic group.
We use this to improve the
Ranestad--Schreyer--Shafiei lower bounds 
for the Waring ranks and cactus ranks of determinants of generic matrices,
Pfaffians of generic skew-symmetric matrices, and determinants of generic symmetric matrices.
\end{abstract}
\maketitle

\section{Introduction}

Let $F \in \C[x_1,\dotsc,x_n]$ be a homogeneous form of degree $d$.
The \defining{Waring rank} of $F$, denoted $r(F)$, is the minimum number of terms
in an expression for $F$ as a linear combination of powers of linear forms:
\[
  F = \sum_{i=1}^r c_i \ell_i(x_1,\dotsc,x_n)^d , \qquad c_i \in \C.
\]
For example,
\[
  x_1 \dotsm x_n = \frac{1}{2^{n-1} n!} \sum_{\substack{\epsilon \in \{\pm1\}^n \\ \epsilon_1 = 1}}
    \left(\prod \epsilon_i\right) \left(\sum \epsilon_i x_i\right)^n,
\]
so $r(x_1 \dotsm x_n) \leq 2^{n-1}$.
In fact $r(x_1 \dotsm x_n) = 2^{n-1}$.
This is a consequence of a general lower bound for Waring rank shown by
Ranestad and Schreyer in  \cite{MR2842085}.
For forms that are invariant under a group action, we improve  the general Ranestad--Schreyer lower bound.

Power sum decompositions of this type and Waring ranks have been studied since the 19th century,
thanks to their connections to the number-theoretic Waring problem,
secant varieties in algebraic geometry,
interpolation and quadrature methods,
mixture models in statistics,
and more.
For comprehensive treatments, including history and applications, see
\cite{MR1735271,MR2865915,comonmour96,MR2447451}.

Unfortunately, Waring ranks are in general difficult to compute, and have been calculated for only
a few families of polynomials.
%
An interesting example of a form whose Waring rank is not yet known is the generic determinant,
\[
  \dett_n = \det \begin{pmatrix} x_{1,1} & \cdots & x_{1,n} \\ \vdots & & \vdots \\ x_{n,1} & \cdots & x_{n,n} \end{pmatrix},
\]
a form of degree $n$ in $n^2$ variables.
Because $\det_2$ is a quadratic form, its rank is known, namely $r(\det_2) = 4$, but $r(\det_n)$ is unknown for $n \geq 3$.
As $\det_n$ is a sum of $n!$ terms of the form $x_{1,i_1} \dotsm x_{n,i_n}$, each with rank $2^{n-1}$,
we have $r(\det_n) \leq 2^{n-1} n!$. (So $r(\det_3) \leq 24$.)
This was recently improved to $r(\det_n) \leq \left(\frac{5}{6}\right)^{\lfloor n/3 \rfloor} 2^{n-1} n!$
\cite[\textsection8]{Derksen:2013sf}. (So $r(\det_3) \leq 20$.)

Several lower bounds for Waring rank have been proposed.
The classical lower bound via Sylvester's catalecticants gives $r(\det_n) \geq \binom{n}{\lfloor n/2 \rfloor}^2$;
this gives $r(\det_3) \geq 9$, and asymptotically (via Stirling approximation) this lower bound grows as $2^n/\sqrt{\pi n/2}$.
In \cite{Landsberg:2009yq} it is shown that $r(\det_n) \geq \binom{n}{\lfloor n/2 \rfloor}^2 + n^2 - (\lfloor n/2 \rfloor + 1)^2$;
this gives $r(\det_3) \geq 14$, but this lower bound has the same asymptotic growth.
Most recently, Shafiei \cite{Shafiei:ud}, using the Ranestad--Schreyer lower bound \cite{MR2842085},
has shown $r(\det_n) \geq \frac{1}{2}\binom{2n}{n}$;
this gives $r(\det_3) \geq 10$, and asymptotically it grows as $4^n/(2\sqrt{\pi n})$.
(Shafiei also considers permanents, Pfaffians, symmetric determinants and permanents, etc.,
see \cite{Shafiei:ud,Shafiei:2013fk}.)

In this paper we show that $r(\det_n) \geq \binom{2n}{n} - \binom{2n-2}{n-1}$;
this gives $r(\det_3) \geq 14$ and asymptotically it is $\frac{3}{2}$ times the Ranestad--Schreyer--Shafiei bound.
This is an example of the main result of this paper,
a lower bound for Waring ranks of invariant forms under the action of a connected group.

For a polynomial $F$, let $\Diff(F)$ be the vector space spanned by the partial derivatives of $F$ of all orders.
A special case of our main result is the following.
\begin{theorem}
Let $G$ be a connected algebraic group and let $V$ be an irreducible representation of $G$.
Let $F \in S^d V$ be an invariant form of degree $d$,
that is, a form such that for every $g \in G$, ${}^g F = F$.
Fix a basis $x_1,\dotsc,x_n$ for $V$, write $F = F(x_1,\dotsc,x_n)$, and let $F' = \d F/\d x_1$.
Then $r(F) \geq \dim \Diff(F) - \dim \Diff(F')$.
\end{theorem}
For example, let $V^* \cong \C^{n^2}$ be the space of $n \times n$ matrices and let
$G = \SL_n \times \SL_n$ act on $V^*$ by left and (inverted) right matrix multiplication.
Then $G$ is connected and $V^*$ is an irreducible representation, as is $V$.
Furthermore $\det_n$ is an invariant form, $\Diff \det_n$ is spanned by the minors of all sizes of the matrix,
and a first partial derivative of $\det_n$ is the determinant $\det_{n-1}$ of the complementary minor.
One can check that $\dim \Diff(\det_n) = \binom{2n}{n}$ and $\dim \Diff(\det_{n-1}) = \binom{2n-2}{n-1}$.
This shows that $r(\det_n) \geq \binom{2n}{n} - \binom{2n-2}{n-1}$.
For more examples, and more details about determinant, see Section~\ref{section: examples}.

Our main theorem, Theorem~\ref{thm: invariant rank bound}, loosens the requirement for $V$ to be an irreducible representation,
generalizes to an invariant subspace of forms instead of a single invariant form,
and actually gives a lower bound for cactus rank (defined below) instead of Waring rank.
Also the statement is made coordinate-free.

Section~\ref{section: preparation} contains some background and basic lemmas.
Our main results are in Section~\ref{section: main results}.
We give examples in Section~\ref{section: examples}.

\section{Preparation}\label{section: preparation}

We review some definitions and basic lemmas.

\subsection{Apolarity}

Let $V$ be a $\C$-vector space with basis $x_1,\dotsc,x_n$,
so $S(V)=\bigoplus_{d=0}^\infty S^dV \cong \C[x_1,\dotsc,x_n]$, where $S(V)$ denotes the symmetric algebra.
We introduce the \defining{dual ring} $S(V^*) \cong \C[\d_1,\dotsc,\d_n]$, where $\d_1,\dotsc,\d_n$
is the dual basis for $V^*$.
Then $S(V^*)$ acts on $S(V)$ by differentiation, where each $\d_i$ acts as $\d/\d x_i$.
This is the \defining{apolar pairing}.
For each degree $d \geq 0$, $S^d V^* \otimes S^d V \to \C$ is a perfect pairing;
for $e > d$, $S^e V^* \otimes S^d V \to 0$ and $S^d V^* \otimes S^e V \to S^{e-d} V$.

Let $F \in S(V)$ be a homogeneous form.
The \defining{apolar} or \defining{annihilating ideal}
$F^\perp \subseteq S(V^*)$ is $F^\perp = \{ \Theta \in S(V^*) \mid \Theta F = 0 \}$,
a homogeneous ideal.
The \defining{apolar algebra} is $A^F = S(V^*) / F^\perp$.
Note that $A^F \cong \Diff(F)$, as vector spaces.
In particular, the length $\ell(A^F)=\dim_{\C} A^F$ is equal to  $\dim \Diff(F)$.

One connection between Waring rank and apolarity is given by the following well-known lemma.
\begin{apolaritylemma}
Let $X \subseteq \PP V$ be a scheme with saturated homogeneous defining ideal $I_X$.
Let $\nu_d : \PP V \to \PP S^d V$ be the degree $d$ Veronese map, and let $[F] \in \PP S^d V$.
Then $[F]$ is in the linear span of $\nu_d(X)$ if and only if $I_X \subseteq F^\perp$.
\end{apolaritylemma}
Here the linear span of a scheme $Z$ is the smallest reduced linear subspace containing $Z$ as a subscheme.
Equivalently it is the linear subspace defined by the degree $1$ part of the ideal $I_Z$.
Note that if $Z = \nu_d(X)$ then the degree $1$ part of $I_Z$ is the degree $d$ part of $I_X$,
interpreted as equations of hyperplanes in $\PP S^d V$.

For proofs of the Apolarity Lemma see for example \cite[Theorem 5.3]{MR1735271}, \cite[\textsection1.3]{MR1780430},
\cite[\textsection4.1]{Teitler:2014gf}.

A scheme $X$ is called \defining{apolar to $F$} if $I_X \subseteq F^\perp$.

Suppose $X = \{[\ell_1],\dotsc,[\ell_r]\} \subseteq \PP V$ is a zero-dimensional reduced scheme.
Then $X$ is apolar to $F$ if and only if $[F]$ is in the linear span of $\nu_d(X) = \{[\ell_1^d],\dotsc,[\ell_r^d]\}$,
equivalently $F = c_1 \ell_1^d + \dotsb + c_r \ell_r^d$ for some scalars $c_i$.
Hence the Waring rank $r(F)$ is the least length of a zero-dimensional {\em reduced} apolar scheme to $F$.
This leads naturally to generalizations; we mention just two:
\begin{enumerate}
\item The \defining{cactus rank} of $F$, denoted $cr(F)$, is the least length of {\em any} zero-dimensional apolar scheme to $F$.
(The name ``cactus rank'' was introduced in \cite{MR2842085}.)
\item The \defining{smoothable rank} of $F$, denoted $sr(F)$,
is the least length of any zero-dimensional smoothable apolar scheme to $F$.
Recall that a scheme is smoothable if it is a flat limit of smooth schemes. (Note that for a scheme of dimension 0, the notions
reduced and smooth are the same.)
\end{enumerate}
Evidently $r(F) \geq sr(F) \geq cr(F)$.

%
%
%

\begin{remark}
Earlier terminology in Definitions 5.1 and 5.66 of \cite{MR1735271} is as follows.
An apolar scheme is also called an \defining{annihilating scheme},
cactus rank is also called \defining{scheme length},
and smoothable rank is also called \defining{smoothable scheme length}.
\end{remark}

\subsection{Lower bounds for rank}

We have remarked that $r(F) \geq sr(F) \geq cr(F)$.
We mention now some well-known lower bounds for rank, only so that we can make comparisons later on
with the lower bound in Theorem~\ref{thm: invariant rank bound}.

The \defining{Sylvester} lower bound for rank is:
\[
  cr(F) \geq \dim A^F_t,
\]
the dimension of the $t$-th graded piece of $A^F$,
for each $0 \leq t \leq d = \deg(F)$.

The \defining{Landsberg--Teitler} lower bound for rank is the following.
Assume that $F$ cannot be written using fewer variables
(that is, if $F \in S^d V'$ and $V' \subseteq V$ then $V' = V$).
Let $\Sigma_t \subseteq V^*$ be the set of points at which $F$ vanishes to order at least $t+1$.
It was shown in \cite{Landsberg:2009yq} that
\[
  r(F) \geq \dim A^F_t + \dim \Sigma_t.
\]

The \defining{Ranestad--Schreyer} lower bound for rank is the following.
Let $\delta$ be an integer such that the apolar ideal $F^\perp$ is generated in degrees less than or equal to $\delta$.
In \cite{MR2842085} it was proven that
\[
  cr(F) \geq \frac{1}{\delta} \ell(A^F).
\]

The simplest way to give an upper bound for Waring rank of a form $F$ is just to exhibit
an explicit expression for $F$ as a sum of powers.
However Bernardi and Ranestad gave an interesting upper bound for cactus rank, as follows.
Let $l$ be a linear form and let $F_l$ be a dehomogenization of $F$ with respect to $l$.
Let $\Diff(F_l) \subseteq S$ be the subspace of $S$ spanned by the partial derivatives of $F_l$ of all orders.
Then $cr(F) \leq \dim \Diff(F_l)$, see \cite[Theorem~1]{MR2996880}.
It was observed by Shafiei (and Pedro Marques) that $\dim \Diff(F_l) \leq \dim \Diff(F) = \ell(A^F)$,
see the remarks following Proposition~3.8 in \cite{Shafiei:ud}.

\subsection{Simultaneous Waring rank}

For a linear series $W \subseteq S^d V$, the \defining{simultaneous Waring rank} $r(W)$ is the least $r$ such that
there exist linear forms $\ell_1,\dotsc,\ell_r \in V$ with $W \subseteq \Span\{\ell_1^d,\dotsc,\ell_r^d\}$.
The \defining{apolar ideal} $W^\perp \subseteq S(V^*)$ is $W^\perp = \bigcap_{F \in W} F^\perp$.
There is an Apolarity Lemma for linear series:
for a scheme $Z \subseteq \PP V$ with vanishing ideal $I_Z\subseteq S(V^*)$, we have $\PP W \subseteq \Span(\nu_d(Z))$ if and only if $I_Z \subseteq W^\perp$.
As before, the simultaneous Waring rank $r(W)$ is the least length of a reduced zero-dimensional apolar scheme,
so we define the \defining{simultaneous smoothable rank} $sr(W)$ to be the least length of a smoothable zero-dimensional apolar scheme,
and the \defining{simultaneous cactus rank} $cr(W)$ to be the least length of a zero-dimensional apolar scheme.
Evidently $r(W) \geq sr(W) \geq cr(W)$.

The \defining{apolar algebra} is defined by $A^W = S(V^*)/W^\perp$.
Let $\Diff(W)$ be the vector subspace of $S(V)$ spanned by all the partial derivatives of all elements of $W$, of all orders.
That is, $\Diff(W) = \sum_{F \in W} \Diff(F)$.
As vector spaces, $A^W \cong \Diff(W)$.
In particular $\ell(A^W) = \dim \Diff(W)$.

Note that if $W = \C F$ is spanned by a single form then $r(W) = r(F)$, $sr(W) = sr(F)$, $cr(W) = cr(F)$,
and $\Diff(W) = \Diff(F)$.
Even though our goal (and main interest) is in providing lower bounds for ranks of single forms,
it turns out to be equally easy to prove the same lower bounds for simultaneous ranks of linear series;
the desired bounds for single forms follow as the special case where $\dim W = 1$.

Each of the lower bounds for rank listed above has an analogue for simultaneous rank.
See sections 2.2, 3.2-3, and 5.2 of \cite{Teitler:2014gf} for a detailed discussion.
Briefly, the Sylvester lower bound is
\[
  cr(W) \geq \dim A^W_t
\]
for each $0 \leq t \leq d$, where $W \subseteq S^d V$.
There is a generalization of the Landsberg--Teitler lower bound,
but it is more complicated and not worth stating here; see \cite[\textsection3.2-3]{Teitler:2014gf}.
The Ranestad--Schreyer lower bound is
\[
  cr(W) \geq \frac{1}{\delta} \ell(A^W),
\]
where $\delta$ is an integer such that $W^\perp$ is generated in degrees less than or equal to $\delta$.

\subsection{Preliminary results}

Here are some easy lemmas.

\begin{notation}
For a nonzero point $p$ in a vector space $W$ we write $[p]$ for the corresponding point in projective space $\PP W$.
Conversely, for a projective variety $X \subset \PP W$, we write $\widetilde{X}$ for the affine cone over $X$.
For a homogeneous ideal $I$, $\VV(I)$ is the projective variety or scheme defined by $I$
and $\wtV(I)$ is the affine variety or scheme defined by $I$.
\end{notation}

If $R$ is a graded ring and $M$ is a graded $R$-module, then $M(d)$ denotes the module $M$ with shifted grading:
$M(d)_e=M_{d+e}$. A homomorphism between graded modules should preserve the grading. If $I$ is an ideal of $R$ and $x\in R$, 
then the colon ideal $(I:x)$ is defined by $(I:x)=\{f\in R\mid fx\in I\}$.
\begin{lemma}
Let $R$ be a graded ring, $I \subseteq R$ a homogeneous ideal, $x \in R$ a homogeneous element of degree $d$.
Then the following is a short exact sequence of graded $R$-modules:
\[
  0 \to (R/(I:x))(-d) \overset{x}{\longrightarrow} R/I \to R/(I+(x)) \to 0.
\]
\end{lemma}
The proof is an easy exercise.

The next lemma is well-known, although usually stated only for the case of a single form ($\dim W = 1$).
\begin{lemma}\label{lemma: colon apolar ideal}
Let $W \subseteq S^d V$ and $\Theta \in S^e V^*$.
Let $\Theta W = \{ \Theta F \mid F \in W \} \subseteq S^{d-e} V$.
Then $(\Theta W)^\perp = (W^\perp : \Theta)$.
\end{lemma}
\begin{proof}
$\Psi \in (\Theta W)^\perp$ if and only if $\Psi(\Theta F) = 0$ for all $F \in W$, if and only if $\Psi\Theta \in W^\perp$.
\end{proof}

Combining these:
\begin{lemma}\label{lemma: length}
Let $W \subseteq S^d V$, $\Theta \in S^e V^*$, and $W' = \Theta W$.
Then we have a short exact sequence of graded $S(V^*)$-modules,
\[
  0 \to A^{W'}(-e) \overset{\Theta}{\longrightarrow} A^W \to S(V^*) / (W^\perp + \Theta) \to 0.
\]
In particular,
\[
  \ell(S(V^*) / (W^\perp + \Theta)) = \ell(A^W) - \ell(A^{W'}) = \dim \Diff(W) - \dim \Diff(W').
\]
\end{lemma}

\section{Main Results}\label{section: main results}

Before giving our main result, we state the following simpler theorem, which does not involve a group action.

\begin{theorem}\label{thm: general derivative}
Let $W \subseteq S^d V$ be a linear series of $d$-forms. There exists a Zariski open dense subset $U$ of $V^*$ such that for 
$\d \in U$ (i.e., $\d$ in general position)  we have
\[
  r(W) \geq sr(W) \geq cr(W) \geq \dim \Diff(W) - \dim \Diff(W'),
\]
where $W'=\partial W=\{\d F\mid F\in W\}\subseteq S^{d-1}V$.
\end{theorem}
\begin{proof}
Let $Z \subseteq \PP V$ be a zero-dimensional apolar scheme to $W$ of length $r$, with defining ideal $I_Z \subseteq F^\perp$.
Since $\d$ is general, the hyperplane $H \subseteq \PP V$ defined by $\d$ is disjoint from $Z$.

Now the affine scheme $\wtV(I_Z) \subseteq V$ is one-dimensional and has no component
contained in the hyperplane $\wtH \subseteq V$.
So $\wtV(I_Z) \cap \wtH$ is supported only at the origin and has length equal to $\ell(Z) = r$.
That is, $\ell(S(V^*) / (I_Z + \d)) = r$.
Since $I_Z \subseteq W^\perp$ we have $\ell(S(V^*) / (W^\perp + \d)) \leq r$.
This shows that $r(W) \geq sr(W) \geq cr(W) \geq \ell(S(V^*) / (W^\perp + \d))$.
By Lemma~\ref{lemma: length},
$\ell(S(V^*) / (W^\perp + \d)) = \ell(A^W) - \ell(A^{W'}) = \dim \Diff(W) - \dim \Diff(W')$.
\end{proof}

For a single form, $W = \C F$, the proof would be the same, just writing $F$ and $F'$ instead of $W$ and $W'$ throughout.

\begin{remark}
The proof of the Ranestad--Schreyer bound uses a dual form $\Theta \in F^\perp$ of degree $\delta$,
and general in the linear series $(F^\perp)_{\delta}$, so that $\Theta$ does not vanish at any point of $Z$.
This requires $\delta$ to be large enough so that $(F^\perp)_{\delta}$ has no basepoints.
In this setting, $F' = \Theta F = 0$.

A general dual linear form $\d \in V^*$ also does not vanish at any point of $Z$.
Lowering the degree from $\delta$ to $1$ accounts for the improvement by a factor of $\delta$ in the above result,
compared to the Ranestad--Schreyer bound.
On the other hand, $F' = \d F \neq 0$.
\end{remark}

Now we state our main theorem.
We make essentially the same argument, except that instead of a general hyperplane we will use
a translation of a given hyperplane by a general group element.

\begin{theorem}\label{thm: invariant rank bound}
Let $G$ be a connected algebraic group and $V$ a representation of $G$.
Let $W \subseteq S^d V$ be an invariant subspace, i.e., ${}^g F \in W$ for all $F \in W$, $g \in G$.
Let $\d \in V^*$ be a nonzero element that is not contained in any proper subrepresentation.
Let $W' = \d W = \{ \d F \mid F \in W \} \subseteq S^{d-1} V$.
Then
\[
  r(W) \geq sr(W) \geq cr(W) \geq \dim \Diff(W) - \dim \Diff(W').
\]
\end{theorem}

\begin{proof}
If the orbit $G \d$ were dense in $V^*$, the result would follow immediately from Theorem~\ref{thm: general derivative},
as ${}^g \d \in V^*$ would be general for $g \in G$ general.
But if the orbit $G \d$ is closed, then a priori the orbit might completely miss the ``general'' open set
of Theorem~\ref{thm: general derivative}.
So we have to argue more directly.

Suppose $Z \subseteq \PP V$ is a zero-dimensional apolar scheme to $W$ of length $\ell(Z) = r$.
Let the support of $Z$ be $\{p_1,\dotsc,p_t\}$.
Let $H \subseteq \PP V$ be the hyperplane defined by $\d = 0$.
We claim that for general $g \in G$, ${}^g Z$ is disjoint from $H$.

The condition $\d$ is not contained in any subrepresentation is equivalent to
requiring that the orbit $G \d$ spans $V^*$.
Then for each $i=1,\dotsc,t$ there is a $g_i \in G$ such that ${}^{g_i} \d$ does not vanish at the point $p_i$;
that is, the hyperplane ${}^{g_i} H$ does not contain $p_i$.
Equivalently, ${}^{g_i^{-1}} p_i \notin H$.
So there is a nonempty open subset $U_i \subseteq G$ such that ${}^g p_i \notin H$ for $g \in U_i$.
Since $G$ is connected (hence irreducible) the intersection $U = \bigcap U_i$ is a nonempty dense open set.
This shows that for general elements $g \in G$ we have ${}^g Z \subseteq \PP V \setminus H$.

Note that $W = {}^g W \subseteq \Span({}^g Z)$ for every $g \in G$.
Now replace $Z$ with ${}^g Z$ for a general $g \in G$.
The rest of the proof is the same as the proof of Theorem~\ref{thm: general derivative}.
Explicitly, let $I_Z \subseteq W^\perp$ be the defining ideal of $Z$.
The affine scheme $\wtV(I_Z)$ is one-dimensional, has degree $r$, and has no component contained in $\wtH$.
Thus $\wtV(I_Z) \cap \wtH$ is supported only at the origin and has length equal to $r$.
Hence $\ell(S(V^*)/(W^\perp + \d)) \leq \ell(S(V^*)/(I_Z + \d)) = r$.
By Lemma~\ref{lemma: length}, $\ell(S(V^*)/(W^\perp + \d)) = \dim \Diff(W) - \dim \Diff(W')$.
\end{proof}

\begin{remark}
If $V$ is an irreducible representation of $G$,
then every nonzero $\d \in V^*$ meets the condition (of not being contained in a subrepresentation).
\end{remark}

Recall that a form $F$ is \defining{semi-invariant} if for every $g \in G$ there is a nonzero scalar $\chi(g)$
such that ${}^g F = \chi(g) F$.
Then $F$ is semi-invariant if and only if its span $W = \C F$ is an invariant subspace.
We obtain
\begin{cor}
Let $G$ be a connected algebraic group and $V$ a representation of $G$.
Let $F \in S^d V$ be a semi-invariant form.
Let $\d \in V^*$ be a nonzero element that is not contained in any proper subrepresentation.
Let $F' = \d F$.
Then
\[
  r(F) \geq sr(F) \geq cr(F) \geq \dim \Diff(F) - \dim \Diff(F').
\]
\end{cor}

Here is a simple and crude lower bound.
\begin{cor}
Let $F \in S^d V$ be a nonzero homogeneous form.
Fix a basis $\{x_1,\dotsc,x_n\}$ for $V$ and the dual basis $\{\d_1,\dotsc,\d_n\}$ for $V^*$.
Let $F = F_k x_1^k + F_{k-1} x_1^{k-1} + \dotsb + F_0$,
$F_i \in \C[x_2,\dotsc,x_n]_{d-i}$.
Suppose either that $\d_1 \in V^*$ is general,
or else that $V$ is a representation of a connected group $G$, $F$ is a semi-invariant of $G$,
and $\d_1$ lies in no proper subrepresentation of $V^*$.
Then $r(F) \geq sr(F) \geq cr(F) \geq \dim \Diff(F_k)$.
\end{cor}
\begin{proof}
Since $F_k = (1/k!) \d_1^k F$ we have $F_k^\perp = (F^\perp : \d_1^k)$
by Lemma~\ref{lemma: colon apolar ideal}.
Then clearly $F^\perp \subseteq (F^\perp : \d_1^k) = F_k^\perp$.
And since $\d_1^{k+1} \in F^\perp$, we have $(\d_1) \subseteq F_k^\perp$ as well
(or: $\d_1 \in F_k^\perp$ since $F_k$ does not involve $x_1$).
So $F^\perp + (\d_1) \subseteq F_k^\perp$,
which implies $\ell(S(V^*) / (F^\perp + \d_1)) \geq \ell(S(V^*) / F_k^\perp) = \dim \Diff(F_k)$.

Finally, we have $cr(F) \geq \ell(S(V^*) / (F^\perp + \d_1))$ in either case,
$\d_1$ is general or $F$ is a semi-invariant,
by the proofs of Theorems~\ref{thm: general derivative} and \ref{thm: invariant rank bound}.
\end{proof}
Compare this with the Bernardi--Ranestad upper bound for cactus rank in terms of dehomogenization:
\[
  \dim \Diff(F_k) \leq cr(F) \leq \dim \Diff(F_k + F_{k-1} + \dotsb + F_0).
\]

\section{Examples}\label{section: examples}

\begin{example}
Let $V^* \cong \C^{n^2}$ be the space of $n \times n$ matrices.
Let $V$ have basis $\{x_{i,j} \mid 1 \leq i, j \leq n\}$ and $V^*$ have the dual basis $\{\d_{i,j}\}$.
Let $\det_n$ be the generic $n \times n$ determinant:
\[
  \dett_n = \det \begin{pmatrix} x_{1,1} & \dotsb & x_{1,n} \\ \vdots & & \vdots \\ x_{n,1} & \dotsb & x_{n,n} \end{pmatrix} ,
\]
so $\det_n \in S^d V$ is a form of degree $n$ in $n^2$ variables.
Recall that the derivatives of $\det_n$ are spanned by minors.
This shows $\dim (S(V^*) / \det_n^\perp)_t = \binom{n}{t}^2$, so
\[
  \ell(A^{\dett_n}) = \sum_{t=0}^n \binom{n}{t}^2 = \binom{2n}{n} .
\]
As mentioned in the introduction, the Sylvester bound shows that the cactus rank and border rank of $\det_n$
are bounded below by $\binom{n}{\lfloor n/2 \rfloor}^2$.
Shafiei has shown that $\det_n^\perp$ is generated by quadrics,
so the Ranestad--Schreyer bound gives $cr(\det_n) \geq \frac{1}{2} \binom{2n}{n}$.

Now $\det_n$ is invariant under the action of $\SL_n \times \SL_n$ on $\C^{n^2}$ by left and (inverted) right matrix multiplication.
And the orbit of $\d_{1,1}$ spans $V^*$; in fact, the subgroup of permutation matrices already takes $\d_{1,1}$
to all the $\d_{i,j}$, a basis for $V^*$.
(In any case, this is an irreducible representation.)
Since $\d_{1,1} \dett_n$ is the complementary $(n-1)$-minor we have
\[
  \ell(S(V^*) / (\d_{1,1}\dett_n)^\perp) = \ell(A^{\dett_{n-1}}) = \binom{2n-2}{n-1} .
\]
Therefore by Theorem~\ref{thm: invariant rank bound},
\[
  r(\dett_n) \geq sr(\dett_n) \geq cr(\dett_n) \geq \binom{2n}{n} - \binom{2n-2}{n-1} .
\]
Note that
\[
  \binom{2n-2}{n-1} = \frac{n^2}{(2n-1)(2n)} \binom{2n}{n} \approx \frac{1}{4} \binom{2n}{n},
\]
so $\binom{2n}{n} - \binom{2n-2}{n-1} \approx \frac{3}{4} \binom{2n}{n}$.
Hence this lower bound is asymptotically $\frac{3}{2}$ times the Ranestad--Schreyer--Shafiei bound.

See Table~\ref{table: determinant bounds} for some values of this bound and a comparison to other bounds.
The upper bound for Waring rank in this table is from \cite[\textsection8]{Derksen:2013sf}.

The upper bound for cactus rank in table~1 is the Bernardi--Ranestad upper bound.
Let $f$ be the dehomogenization of $\det_n$ with respect to $x_{n,n}$.
Then $\dim \Diff(f) = \dim \Diff(\det_n) - 2$;
indeed, derivatives of $f$ are obtained by dehomogenizing the corresponding derivatives of $\det_n$,
except that $\d_{n,n} f = 0$ and $x_{n,n} \notin \Diff(f)$.
So $cr(\det_n) \leq \dim \Diff(\det_n) - 2 = \binom{2n}{n} - 2$.
\begin{table}[htb]
\begin{tabular}{l rrrrrrr}
\toprule
$n$                                     & $2$ &  $3$ &  $4$ &   $5$ &   $6$ &    $7$ &    $8$ \\
\midrule
Sylvester                               & $4$ &  $9$ & $36$ & $100$ & $400$ & $1225$ & $4900$ \\
Landsberg--Teitler                      & $4$ & $14$ & $43$ & $116$ & $420$ & $1258$ & $4939$ \\
Ranestad--Schreyer--Shafiei             & $3$ & $10$ & $35$ & $126$ & $462$ & $1716$ & $6435$ \\
Theorem~\ref{thm: invariant rank bound} & $4$ & $14$ & $50$ & $182$ & $672$ & $2508$ & $9438$ \\
\midrule
Upper bound for $cr(\det_n)$            & $4$ & $18$ &  $68$ &  $250$ &   $922$ &   $3430$ &   $12868$ \\
Upper bound for $r(\det_n)$             & $4$ & $20$ & $160$ & $1600$ & $16000$ & $224000$ & $3584000$ \\
\bottomrule
\end{tabular}
\caption{Comparison of bounds for rank of determinant}\label{table: determinant bounds}
\end{table}
\end{example}

\begin{example}
Let $X$ be a generic $(2n) \times (2n)$ skew-symmetric matrix,
that is, $X = (x_{i,j})$ such that $x_{i,j} = -x_{j,i}$ and $x_{i,i}=0$.
The Pfaffian of $X$ is a polynomial of degree $n$ in the entries of $X$, which we denote $\pf_n$ or $\pf(X)$,
with the property that $\pf_n^2 = \det(X)$.
For $n=1,2$ we have
\[
  \pf_1 = \pf \begin{pmatrix} 0 & x_{1,2} \\ -x_{1,2} & 0 \end{pmatrix} = x_{1,2},
\]
and
\[
  \pf_2 = \pf \begin{pmatrix}
    0 & x_{1,2} & x_{1,3} & x_{1,4} \\
    -x_{1,2} & 0 & x_{2,3} & x_{2,4} \\
    -x_{1,3} & -x_{2,3} & 0 & x_{3,4} \\
    -x_{1,4} & -x_{2,4} & -x_{3,4} & 0
    \end{pmatrix}
    = x_{1,2} x_{3,4} - x_{1,3}x_{2,4} + x_{1,4}x_{2,3}.
\]
In general,
\[
  \pf_n = \sum (-1)^{\sigma} x_{\sigma(1),\sigma(2)} \dotsm x_{\sigma(2n-1),\sigma(2n)} ,
\]
the sum over permutations $\sigma \in S_{2n}$ such that $\sigma(2i-1) < \sigma(2i)$ for all $i$
and $\sigma(1) < \sigma(3) < \dotsb < \sigma(2n-1)$,
equivalently over unordered partitions of $\{1,\dotsc,2n\}$ into pairs.
Note, there are $(2n-1)!! = (2n)!/(2^n n!)$ such partitions.
There is a ``Laplace expansion'': for each $j$, $1 \leq j \leq n$,
\[
  \pf(X) = \sum_{i<j} (-1)^{i+j+1} x_{i,j} \pf(X^{i,j}) + \sum_{i > j} (-1)^{i+j} x_{i,j} \pf(X^{i,j}) ,
\]
where $X^{i,j}$ is the matrix obtained by deleting the $i$th and $j$th rows and columns of $X$.
See \cite{MR0453723,MR0257105,MR554859}.

Note that $\pf_n$ is invariant under the conjugation action of $\SO_{2n}$
on the space of skew-symmetric matrices.
This is an irreducible representation of a connected group,
so Theorem~\ref{thm: invariant rank bound} applies.

By the Laplace expansion, $\d_{n-1,n} \pf_n$ is the Pfaffian of the $(2n-2) \times (2n-2)$ skew-symmetric matrix
obtained by deleting the $(2n-1)$-st and $(2n)$-th rows and columns of $X$;
we may regard $\d_{n-1,n} \pf_n = \pf_{n-1}$.

More generally, all the derivatives of $\pf_n$ are spanned by Pfaffians of even-sized
principal (i.e., skew-symmetric) submatrices of $X$.
So
\[
  \dim A^{\pf_n}_t = \binom{2n}{2t}
\]
for all $0 \leq t \leq n$.
Therefore the Sylvester bound shows that the cactus rank and border rank of $\pf_n$
are bounded below by $\binom{2n}{2 \lfloor n/2 \rfloor}$.
And
\[
  \ell(A^{\pf_n}) = \sum_{t=0}^n \binom{2n}{2t} = 2^{2n-1} .
\]
Shafiei has shown that $\pf_n^\perp$ is generated by quadrics \cite[Theorem~4.11]{Shafiei:ud},
so by the Ranestad--Schreyer bound, $cr(\pf_n) \geq 2^{2n-2}$.
Finally, Theorem~\ref{thm: invariant rank bound} gives
\[
  r(\pf_n) \geq sr(\pf_n) \geq cr(\pf_n) \geq \ell(A^{\pf_n}) - \ell(A^{\pf_{n-1}}) = 2^{2n-1} - 2^{2n-3} = 3 \cdot 2^{2n-3}.
\]
This is exactly $\frac{3}{2}$ times the Ranestad--Schreyer--Shafiei bound.

Some values of these bounds are shown in Table~\ref{table: pfaffian bounds}.
The upper bound for Waring rank in this table comes from the expression of
$\pf_n$ as a sum of $(2n)!/2^n n!$ terms each of rank $2^{n-1}$.

The upper bound for cactus rank is Shafiei's loosening of the Bernardi--Ranestad upper bound:
$cr(\pf_n) \leq \ell(A^{\pf_n}) = 2^{2n-1}$.
This can be improved by considering a dehomogenization.
(In particular, this loosening is the reason we get a worse bound for $cr(\pf_2)$ than the
bound for Waring rank $r(\pf_2)$.)

\begin{table}[htb]
\begin{tabular}{l rrrrrrr}
\toprule
$n$                                     & $2$ &  $3$ &  $4$ &   $5$ &   $6$ &    $7$ &    $8$ \\
\midrule
Sylvester                               & $6$ & $15$ & $70$ & $210$ &  $924$ & $3003$ & $12870$ \\
Ranestad--Schreyer--Shafiei             & $4$ & $16$ & $64$ & $256$ & $1024$ & $4096$ & $16384$ \\
Theorem~\ref{thm: invariant rank bound} & $6$ & $24$ & $96$ & $384$ & $1536$ & $6144$ & $24576$ \\
\midrule
Upper bound for $cr(\pf_n)$             & $8$ & $32$ & $128$ &   $512$ &   $2048$ &    $8192$ &     $32768$ \\
Upper bound for $r(\pf_n)$              & $6$ & $60$ & $840$ & $15120$ & $332640$ & $8468640$ & $259459200$ \\
\bottomrule
\end{tabular}
\caption{Comparison of bounds for rank of Pfaffian}\label{table: pfaffian bounds}
\end{table}
\end{example}

\begin{example}
Let $\sdet_n$ be the determinant of a generic symmetric matrix,
that is a matrix $X = (x_{i,j})$ such that $x_{i,j} = x_{j,i}$.
For example $\sdet_1 = x_{1,1}$,
\[
  \sdet_2 = \det \begin{pmatrix} x_{1,1} & x_{1,2} \\ x_{1,2} & x_{2,2} \end{pmatrix} = x_{1,1} x_{2,2} - x_{1,2}^2 ,
\]
and so on.

Note that $\sdet_n$ is invariant under the action of $\SL_n$
on the space of symmetric matrices given by $(A,M) \mapsto A M A^t$.
This is an irreducible representation of a connected group,
so Theorem~\ref{thm: invariant rank bound} applies.

Shafiei has shown
\[
  \dim A^{\sdet_n}_t = \frac{1}{n+1} \binom{n+1}{t} \binom{n+1}{t-1},
\]
see \cite[Lemma~2.5]{Shafiei:2013fk}.
Therefore the Sylvester bound shows that the cactus rank and border rank of $\sdet_n$
are bounded below by $\frac{1}{n+1} \binom{n+1}{\lceil n/2 \rceil} \binom{n+1}{\lceil n/2 \rceil - 1}$.
And
\[
  \ell(A^{\sdet_n}) = \frac{1}{n+2} \binom{2n+2}{n+1},
\]
the $(n+1)$-st Catalan number $C_{n+1}$, see \cite[Corollary~2.6]{Shafiei:2013fk}.
Shafiei has shown that $\sdet_n$ is generated by quadrics \cite[Theorem~3.11]{Shafiei:2013fk}.
So by the Ranestad--Schreyer bound, $cr(\sdet_n) \geq \frac{1}{2(n+2)} \binom{2n+2}{n+1}$,
see \cite[Theorem~4.6]{Shafiei:2013fk}.

Since $\d_{n,n} \sdet_n = \sdet_{n-1}$, Theorem~\ref{thm: invariant rank bound} gives
\[
  r(\sdet_n) \geq sr(\sdet_n) \geq cr(\sdet_n) \geq C_{n+1} - C_n = \frac{1}{n+2}\binom{2n+2}{n+1} - \frac{1}{n+1}\binom{2n}{n}.
\]
This is asymptotically $\frac{3}{2}$ times the Ranestad--Schreyer--Shafiei bound.

Some values of these bounds are shown in Table~\ref{table: symmetric determinant bounds}.
The upper bound for cactus rank is Shafiei's loosening of the Bernardi--Ranestad upper bound,
$cr(\sdet_n) \leq \ell(A^{\sdet_n}) = \frac{1}{n+2}\binom{2n+2}{n+1}$;
this can be improved by considering a dehomogenization.

\begin{table}[htb]
\begin{tabular}{l rrrrrrr}
\toprule
$n$                                     & $2$ &  $3$ &  $4$ &   $5$ &   $6$ &    $7$ &    $8$ \\
\midrule
Sylvester                               & $3$   &  $6$ & $20$ & $50$ & $175$   &  $490$ & $1764$ \\
Ranestad--Schreyer--Shafiei             & $2.5$ &  $7$ & $21$ & $66$ & $214.5$ &  $715$ & $2431$ \\
Theorem~\ref{thm: invariant rank bound} & $3$   &  $9$ & $28$ & $90$ & $297$   & $1001$ & $3432$ \\
\midrule
Upper bound for $cr(\sdet_n)$             & $5$ & $14$ & $42$ & $132$ & $429$ & $1430$ & $4862$ \\
\bottomrule
\end{tabular}
\caption{Comparison of bounds for rank of symmetric determinant}\label{table: symmetric determinant bounds}
\end{table}
\end{example}

\begin{example}
Let $V^*$ be the space of $m \times n$ matrices, $m \leq n$.
Fix a basis $\{x_{i,j} \mid 1 \leq i \leq m, 1 \leq j \leq n\}$ for $V$ and the dual basis $\{\d_{i,j}\}$ for $V^*$.
Let $X$ be the generic $m \times n$ matrix $X = (x_{i,j})$.
Fix an integer $d$, $1 \leq d \leq m \leq n$, and let $W_d$ be the linear series spanned by the $d$-minors of $X$.
Then $\Diff(W_d)$ is spanned by all the $t$-minors of $X$ for $0 \leq t \leq d$, so
\[
  \dim \Diff(W_d) = \sum_{t=0}^d \binom{m}{t} \binom{n}{t} = \sum_{t=0}^d \binom{m}{m-t} \binom{n}{t} .
\]
In particular if $d = m \leq n$ (the case of maximal minors) then $\dim \Diff(W_m) = \binom{m+n}{m}$.

Shafiei's results imply that $W_d^\perp$ is generated by quadrics.
The Ranestad--Schreyer lower bound thus gives
\[
  cr(W_d) \geq \frac{1}{2} \sum_{t=0}^d \binom{m}{t} \binom{n}{t},
\]
in particular if $d = m \leq n$ then $cr(W_m) \geq \frac{1}{2} \binom{m+n}{n}$.

Now, $W_d$ is invariant under the action of $\SL_m \times \SL_n$ on $V^*$ by left and (inverted) right matrix multiplication.
Note that $W' = \d_{1,1} W_d$ is exactly the linear series of $(d-1)$-minors of the $(m-1) \times (n-1)$ matrix obtained by deleting
the first row and column of $X$.
So
\[
  \dim \Diff(W') = \sum_{t=0}^{d-1} \binom{m-1}{t} \binom{n-1}{t},
\]
in particular if $d = m \leq n$ then $\dim \Diff(W') = \binom{m+n-2}{m-1}$.
Thus
\[
  cr(W_d) \geq \sum_{t=0}^d \left(\binom{m}{t} \binom{n}{t} - \binom{m-1}{t-1} \binom{n-1}{t-1} \right)
\]
in general
(with the understanding $\binom{m-1}{-1} = \binom{n-1}{-1} = 0$),
and for maximal minors, if $d = m \leq n$ then
\[
  cr(W_m) \geq \binom{m+n}{m} - \binom{m+n-2}{m-1} .
\]
\end{example}

\begin{example}
Let $V_1\cong \C^{pq}$ be the space of $p\times q$ matrices, $V_2\cong \C^{qr}$ be the space of $q\times r$ matrices,
and $V_3 \cong \C^{rp}$ be the space of $r \times p$ matrices.
We choose a basis $\{x_{i,j}\}_{1\leq i\leq p,1\leq j\leq q}$ of $V_1$, a basis $\{y_{i,j}\}_{1\leq i\leq q,1\leq j\leq r}$ of $V_2$,
and a basis $\{z_{i,j}\}_{1 \leq i \leq r, 1 \leq j \leq p}$ of $V_3$.
Consider the tensor
\[
  T_{p,q,r}=\sum_{i,j,k}x_{i,j}y_{j,k}z_{k,i}\in \C^{pq}\otimes \C^{qr}\otimes \C^{rp}.
\]
This represents the \emph{matrix multiplication} map $V_1 \otimes V_2 \to V_3$.
We will give a lower bound for the tensor rank of $T_{p,q,r}$.

Set $V=V_1\oplus V_2$.
Let $W\subseteq S^2V$ be the space spanned by all $\sum_{j=1}^q x_{i,j}y_{j,k}$ with $1\leq i\leq p$ and $1\leq k\leq r$.
Note that these are the entries of the $p \times r$ matrix obtained by multiplying the matrices
$(x_{i,j})_{1 \leq i \leq p, 1 \leq j \leq q}$ and $(y_{i,j})_{1 \leq i \leq q, 1 \leq j \leq r}$.

The tensor rank $tr(T_{p,q,r})$ of the tensor $T_{p,q,r}$ is the smallest number $m$
for which there exist pure tensors $f_i(x,y)=g_i(x)h_i(y)\in V_1\otimes V_2$, $i=1,2,\dots,m$
such that $W$ is contained in the span of $f_1(x,y),\dots,f_m(x,y)$ \cite[Thm.~3.1.1.1]{MR2865915}.
We can write $f_i(x,y)=\frac{1}{4}((g_i(x)+h_i(y))^2-(g_i(x)-h_i(y))^2)$, so $tr(T_{p,q,r})\geq \frac{1}{2}r(W)$.
On the other hand, if $W$ is contained in the space spanned by $u_1(x,y)^2,\dots,u_m(x,y)^2$ where $u_1,\dots,u_m$ are linear, then we can write $u_i(x)=v_i(x)+w_i(y)$
and $W$ is contained in the span of $v_i(x)w_i(y)$, $i=1,2,\dots,m$ because $W$ consists of forms that are bilinear in $x$ and $y$. This shows that $tr(T_{p,q,r})\leq r(W)$.
We conclude that
\[
  \textstyle \frac{1}{2}r(W)\leq tr(T_{p,q,r})\leq r(W).
\]

There is a natural action of $G=\GL_p\times \GL_q\times \GL_r$ on $V$.
The space $W$ is invariant under the action of $G$.
Let $\partial=\frac{\partial}{\partial x_{1,1}}+\frac{\partial}{\partial y_{1,1}}\in V^*$. The space $V^*\cong V_1^*\oplus V_2^*$ is the sum of two non-isomorphic irreducible representations of $G$.
The only proper invariant subspaces of $V^*$ are $V_1^*$ and $V_2^*$ and $\partial$ does not lie in either of them.
The space $\Diff(W)$ is spanned by $W$, all $x_{i,j}$, all $y_{i,j}$, and $\C$.
The dimension of $W$ is $pr$ while the span of all $x_{i,j}$ and all $y_{i,j}$ has dimension $pq+qr$.
So the dimension of $\Diff(W)$ is $pq+qr+pr+1$.
The space $W'$ is spanned by all $y_{1,k}$ for $k>1$, $x_{i,1}$ for $i>1$ and $x_{1,1}+y_{1,1}$.
The dimension of $W'$ is $(p-1)+(r-1)+1=p+r-1$, and the dimension of $\Diff(W')$ is  $p+r$.
We have 
\begin{multline*}
r(W)\geq sr(W)\geq cr(W)\geq \dim \Diff(W)-\dim \Diff(W')=\\ =(pq+qr+pr+1)-(p+r)=pq+qr+pr-p-r+1.
\end{multline*}

It follows that
\[
  tr(T_{p,q,r})\geq\textstyle \frac{1}{2}(pq+qr+pr-p-r+1).
\]
If $p=q=r=n$ then we get $r(W)\geq 3n^2-2n+1$. It follows that
\[
  tr(T_{n,n,n})\geq \textstyle \frac{3}{2}n^2-n+\frac{1}{2}.
\]
This is a nontrivial lower bound for the tensor rank of $T_{n,n,n}$, but there are better bounds known: Recently, Massarenti and Raviolo (\cite{MR,MR2,MR3}) improved a
lower bound of Landsberg (\cite{Landsberg2}). Both lower bounds grow asymptotically as $3n^2-o(n^2)$.
From these bounds we also get a lower bound
for $r(W)$, namely $r(W)\geq 3n^2-o(n^2)$, which is slightly worse than our bound.
\end{example}

\bibliographystyle{amsplain}
\bibliography{../../biblio}

\bigskip

\end{document}